\documentclass[twoside,11pt]{article} 
\usepackage{jmlr2e} 

\title{Sharp oracle bounds for monotone and convex regression through aggregation}
\ShortHeadings{Monotone and convex regression through aggregation}{Bellec and Tsybakov}
\jmlrheading{1}{2000}{1-48}{4/00}{10/00}{Pierre C. Bellec and Alexandre B. Tsybakov}
\editor{Vladimir Vapnik, Alex Gammerman, and Vladimir Vovk}

\usepackage{hyperref}
\usepackage{cleveref}

\usepackage{filecontents}

\usepackage[T1]{fontenc}
\usepackage{lmodern}
\usepackage{graphicx}
\usepackage{algorithm2e}
\usepackage{subfigure}
\usepackage{ifxetex,ifluatex}
\usepackage{fixltx2e} 
\IfFileExists{microtype.sty}{\usepackage{microtype}}{}
\IfFileExists{upquote.sty}{\usepackage{upquote}}{}
\ifnum 0\ifxetex 1\fi\ifluatex 1\fi=0 
  \usepackage[utf8]{inputenc}

\usepackage{enumitem}

\usepackage{thmtools}
\usepackage{mathtools}

\usepackage{algorithmicx}
\hypersetup{breaklinks=true,
    linkcolor=green,
    citecolor=magenta,
    colorlinks=true,
pdfborder={0 0 0}}

\usepackage{enumitem}

\usepackage{thmtools}
\usepackage{mathtools}

\numberwithin{equation}{section}

\DeclareMathOperator*{\argmin}{argmin}

\newcommand{\design}{\mathbb{X}}

\newcommand{\E}{\mathbb{E}_{\vmu}}
\newcommand{\R}{\mathbb{R}}
\newcommand{\Rn}{\mathbb{R}^n}

\newcommand{\vmu}{\boldsymbol{\mu}}
\newcommand{\vtheta}{{\boldsymbol{\theta}}}
\newcommand{\htheta}{{\boldsymbol{\hat \theta}}}
\newcommand{\vu}{\boldsymbol{u}}
\newcommand{\hmu}{\boldsymbol{\hat \mu}}
\newcommand{\vbeta}{\boldsymbol{\beta}}

\newcommand{\vy}{\mathbf{y}}
\newcommand{\vxi}{{\boldsymbol{\xi}}}

\newcommand{\increasings}{\mathcal{S}^\uparrow}
\newcommand{\convexs}{\mathcal{S}^{\rm C}}

\newcommand{\scalednorm}[1]{\left\Vert  #1 \right\Vert}
\newcommand{\scalednorms}[1]{\scalednorm{ #1 }^2}

\newcommand{\zeronorm}[1]{| #1 |_0}

\newcommand{\vomega}{{\boldsymbol{\omega}}}

\begin{document}

\author{\name Pierre C. Bellec \email pierre.bellec@ensae.fr\\ 
    \and
    \name Alexandre B. Tsybakov \email alexandre.tsybakov@ensae.fr\\ 
    \addr ENSAE,\\ 
    3 avenue Pierre Larousse\\ 
92240 Malakoff, France}
\maketitle

\begin{abstract}
    We derive oracle inequalities 
    for the problems of isotonic and convex regression
    using the combination of $Q$-aggregation procedure and sparsity pattern aggregation.
    This improves upon the previous results
    including the oracle inequalities for the constrained least squares estimator.
    One of the improvements is that our oracle inequalities 
 are sharp, i.e., 
with leading constant  1.
It allows us to obtain bounds for the minimax regret thus accounting for model misspecification, which was not possible based on the previous results. 
Another improvement is that we obtain oracle inequalities both with high probability and in expectation.
\end{abstract}

\section{Introduction}
Assume that we have the observations
\begin{equation}
    \label{eq:observations}
    Y_i = \mu_i + \xi_i, \qquad i=1,...,n,
\end{equation}
where $\vmu =(\mu_1,...,\mu_n)^T \in\Rn$ is unknown,
$\vxi = (\xi_1,...,\xi_n)^T$ is a noise vector with $n$-dimensional Gaussian distribution
$\mathcal{N}(0,\sigma^2 I_{n\times n})$ where $\sigma>0$.
We observe $\vy = (Y_1,...,Y_n)^T$ and we want to estimate $\vmu$. We can interpret $\mu_i$ as the values $f(X_i)$ of an unknown regression function $f:{\cal X}\to \R$ at given non-random points $X_i\in{\cal X}$, $i=1,\dots,n$, where ${\cal X}$ is an abstract set. Then, the equivalent setting is that we observe $\vy$ along with $(X_1,\dots,X_n)$ but the values of $X_i$ are of no interest and can be replaced by their indices if we  measure the loss in a discrete norm. Namely, for any $\vu\in\Rn$ we consider
 the scaled (or the empirical) norm $\Vert\cdot\Vert$ defined by 
 \begin{equation}
     \scalednorms{\vu} = \frac{1}{n} \sum_{i=1}^n u_i^2.
\end{equation}
We will measure the error  of an estimator $\hmu$ of $\vmu$ 
by the distance $\|\hmu - \vmu\|$. 
Let $\increasings$ be the set of all non-decreasing sequences:
\begin{equation}
    \increasings \coloneqq \{\vu=(u_1,...,u_n)\in\Rn:  u_i \le u_{i+1}, \quad i=1,...,n-1\}.
\end{equation}
For a subset $\cal S$ of $\increasings$, and any $\vmu\in \Rn$ the quantity $\min_{\vu\in {\cal S}}\Vert\vu - \vmu\Vert$
is the smallest approximation error achievable by a sequence in the set ${\cal S}$.
This quantity defines a benchmark or oracle performance on ${\cal S}$. The accuracy of an estimator $\hmu$ with respect to the oracle for any $\vmu$, not necessarily $\vmu\in {\cal S}$, can be characterized by the excess loss $\|\hmu - \vmu\| - \min_{\vu\in {\cal S}}\Vert\vu - \vmu\Vert$. This is a measure of performance of $\vmu$ under model misspecification. One can also consider the expected quantities $R_1(\hmu,\vmu)=\E\|\hmu - \vmu\| - \min_{\vu\in {\cal S}}\Vert\vu - \vmu\Vert$ or $R_2(\hmu,\vmu)=\E\|\hmu - \vmu\|^2 - \min_{\vu\in {\cal S}}\Vert\vu - \vmu\Vert^2$ known under the name of regret measures. Here, $\E$ denotes the expectation with respect to the distribution of $\vy$ satisfying \eqref{eq:observations}. The minimax regret is defined as $\min_{\hmu}\max_{\vmu\in \Rn}R_i(\hmu,\vmu)$ for $i=1,2$, where $\min_{\hmu}$ denotes the minimum over all estimators. We can characterize the performance of an estimator 
$\tilde\vmu$ by the closeness of its maximal regret $\max_{\vmu\in \Rn}R_i(\tilde\vmu,\vmu)$ to the minimax regret. 
This approach to measure the performance of estimators under model misspecification was pioneered by Vapnik and Chervonenkis who called it the criterion of minimax of the loss~\cite[Chapter 6]{VapnikChervonenkis}. In this paper, we follow this approach and establish non-asymptotic bounds for the maximal regret for some classes ${\cal S}$ of monotone and convex functions. 

When the model is well-specified, i.e., the true function $\vmu$ belongs to the class ${\cal S}$, the approximation error vanishes and instead of the minimax regret it is natural to consider the minimax risk defined either as $\min_{\hmu}\max_{\vmu\in{\cal S}}\E\|\hmu - \vmu\|$ or as $\min_{\hmu}\max_{\vmu\in{\cal S}}\E\|\hmu - \vmu\|^2$ (the minimax squared risk). It is easy to see that the minimax risk is not greater than the minimax regret. A classical problem in nonparametric statistics is to study the behavior of minimax risks for different classes ${\cal S}$. In particular, there exist results concerning the minimax risks for classes of monotone and convex functions in our setting. We review some of them below. The behavior of the minimax regret is much less studied. For a recent overview and some general results we refer to \cite{RakhlinSridharanTsybakov} where it is shown that the rate of minimax regret can be different from that of the minimax risk.  Note that \cite{RakhlinSridharanTsybakov} studies the prediction problem with i.i.d. observations, which is a setting different from ours.  

A well-studied estimator under the monotonicity and convexity assumptions is the least squares estimator 
\begin{equation}
\label{eq:lse-increasing}
    \hmu^{LS}({\cal S}) \in \argmin_{\vu\in{\cal S}} \scalednorms{\vy - \vu}.
\end{equation}
In \cite{NemirovskiPolyakTsybakov} it was shown that $\hmu^{LS}({\cal S})$ attains, up to logarithmic factors, the rates $n^{-2/3}$ and $n^{-4/5}$ of the mean squared risk for  classes ${\cal S}$ of monotone and convex functions respectively and that these rates are optimal up to logarithmic factors when the minimax squared risk is used as a criterion. 
Under monotonicity constraints, the rate $n^{-2/3}$ 
was later observed in different settings, see for instance \cite{banerjee2001likelihood,balabdaoui2007estimation}.

One class of monotone functions we will be interested in here is defined~as $$\increasings(V)=\{\vmu\in\increasings:\ V(\vmu)\le V\}$$ where $V(\vmu) = \mu_n - \mu_1$ for any $\vmu=(\mu_1,\dots,\mu_n)\in\increasings$, and $V>0$ is a given constant. In \cite{meyer2000degrees,zhang2002risk} it was shown that 
 for any $\vmu\in\increasings$ we have
    \begin{equation}
        \E \scalednorms{\hmu -  \vmu} \le
        c
        \max\left(
            \left(\frac{\sigma^2 V(\vmu)}{n}\right)^{2/3}
            ,\frac{\sigma^2 \log n}{n}
        \right)
        \label{eq:lse-n23}
    \end{equation}
for $\hmu= \hmu^{LS}(\increasings)$ and some absolute constant $c>0$.  This immediately implies an upper bound on the minimax risk on $\increasings(V)$. A recent paper \cite{chatterjee2013risk} establishes the oracle inequality
\begin{equation}
    \E \scalednorms{ \hmu^{LS}(\increasings)-\vmu} \le C_* \min_{\vu\in\increasings}
    \left(
        \scalednorms{\vmu - \vu} + \frac{c_*\sigma^2k(\vu)}{n} \log \frac{en}{k(\vu)}    \right)     
    \label{eq:lse-oracle-ineq}
\end{equation}
valid for  all $\vmu\in \increasings$ 
where either $C_*=6, c_*=1$  \cite[inequality (18)]{chatterjee2013risk} or 
$C_*=4, c_*=4$ \cite[inequality (30)]{chatterjee2013risk}.
Here, $k(\vu)\ge 1$ for $\vu=(u_1,\dots,u_n) \in \increasings$ is the integer such that $k(\vu)-1$ is the number
of inequalities $u_i\le u_{i+1}$ that are strict for $i=1,\dots,n-1$ (number of jumps of $\vu$).
Inequality \eqref{eq:lse-oracle-ineq} implies (up to a logarithmic factor) a bound as in  \eqref{eq:lse-n23} and also gives some more 
insight into the problem. For example, \eqref{eq:lse-oracle-ineq} shows that the fast rate $\frac{\log n}{n}$ is achieved if $\vmu$ has only one jump or a fixed, independent of $n$, number of jumps. This is not granted by  \eqref{eq:lse-n23}.

 Along with the least squares estimator, one may consider estimation of monotone functions via penalized least squares with total variation penalty. The corresponding estimator $\hmu^{TV}$ is defined as
 \begin{equation}
    \hmu^{TV} \in \argmin_{\vu\in\Rn}\left(
    \frac{1}{2} \scalednorms{\vu - \vy}
    + \lambda \sum_{i=1}^{n-1} |u_{i+1}-u_i|\right),
    \label{eq:def-tv}
\end{equation}
where $\lambda > 0$ is a tuning parameter.
 Statistical properties of this estimator were first studied in \cite{mammen1997locally} where it was shown that $\|\hmu^{TV}-\vmu\|$ attains the optimal rate $n^{-1/3}$ in probability on the class of functions of bounded variation (and thus on $\increasings(V)$).
Recently, the performance of $\hmu^{TV}$
was analyzed in \cite{dalalyan2014prediction}
by considering $\hmu^{TV}$ as a special instance of the Lasso estimator.
If $\vmu^\uparrow$ is the projection of $\vmu$ onto $\increasings$,
$\delta\in(0,1)$ is a constant,
and the tuning parameter $\lambda$ is given by
\begin{equation}
    \lambda =
    \sigma \sqrt{\frac{\log(n/\delta)}{k^* n}}
    \qquad \text{where } k^* = \left(
        \frac{V(\vmu^{\uparrow})^2 n \log(n/\delta)}{\sigma^2}
    \right)^{1/3},
\end{equation}
the estimator $\hmu^{TV}$
satisfies with probability greater than $1-2 \delta$
the following oracle inequality \cite[Proposition 6]{dalalyan2014prediction}:
\begin{eqnarray} \label{eq:soi-lasso-n23}
    \scalednorms{\hmu^{TV} - \vmu}
    &\le&
    \scalednorms{\vmu^{\uparrow} - \vmu}
    + 6 \left(
        \frac{\sigma^2 V(\vmu^\uparrow) \sqrt{\log(n/\delta)} }{n}
    \right)^{2/3}
    \\
    && +
    \frac{2\sigma^2(1+2\log(1/\delta))}{n}
   \nonumber
\end{eqnarray}
for all $\vmu\in \Rn$. It follows from \eqref{eq:soi-lasso-n23} that if the tuning parameter is chosen correctly,
the estimator $\hmu^{TV}$
achieves, up to a logarithmic factor,  the minimax rate $n^{-2/3}$ in probability
on the class $\increasings(V)$. Also, \eqref{eq:soi-lasso-n23} implies a bound for 
the excess losses $\|\hmu^{TV} - \vmu\|^i - \min_{\vu\in \increasings(V)}\Vert\vu - \vmu\Vert^i$, $i=1,2$, corresponding to the class $\increasings(V)$. However, \eqref{eq:soi-lasso-n23} does not allow us to evaluate the expected 
regrets $R_i(\hmu^{TV},\vmu)$ since $\hmu^{TV}$ depends on $\delta$.
It is also shown in \cite[Proposition 4]{dalalyan2014prediction} that if $\lambda = 2\sigma \sqrt{(2/n)\log(n/\delta)}$,
the estimator $\hmu^{TV}$ satisfies
\begin{equation}
    \scalednorms{\hmu^{TV} - \vmu}
    \le
    \min_{\vu\in\Rn}
    \left(
        \scalednorms{\vu - \vmu}
        + \frac{4 \sigma^2 k(\vu) \log(n/\delta)}{n} r_n(\vu)    \right)
    \label{eq:soi-lasso-sparse}
\end{equation}
with probability greater than $1-2\delta$,
where $k(\vu)-1$ for $\vu\in\Rn$ is the number of jumps of $\vu$, i.e., the cardinality of the set
$\{i\in \{1,...,n-1\}: \, u_i\ne u_{i+1}\}$, $r_n(\vu) = 3 + 256(\log(n) + (n/\Delta(\vu)))$
and $\Delta(\vu)$ is the minimum distance between two jumps in the sequence $\vu$:
\begin{equation*}
  \Delta(\vu) =
    \min \left\{
        d\ge 1: \ \exists k\in\{1,...,n\} \text{ with }
        u_{k+1}\ne u_k \text{ and } u_{k+d+1}\ne u_{k+d}
    \right\}.
    \end{equation*}
 
 The expressions on the right hand sides of \eqref{eq:lse-oracle-ineq} and \eqref{eq:soi-lasso-sparse}
are small if the unknown sequence $\vmu$ is well approximated
by a piecewise constant sequence with not too many pieces.
In this regard, the two bounds have some similarity to sparsity oracle inequalities
in high-dimensional linear regression, cf. \cite{RigolletTsybakov2011,rigollet2012sparse,tsybakovICM}. This similarity can be easily explained as follows.  
Write \eqref{eq:observations} in the equivalent form 
$$
\vy= \design\vbeta^* + \vxi,
$$
with the matrix $\design=(X_{ij})_{i=1,...,n,\;j=1,...,n}$
where $X_{ij}= 1$ if $j\le i$ and $X_{ij}= 0$ otherwise, and $\vbeta^*=(\beta^*_1,\dots,\beta^*_n)$ where $\beta^*_1 = \mu_1$ and  $\beta^*_i = \mu_i - \mu_{i-1}$ for $i=2,\dots,n$. With this notation, $k(\vmu) \in  \{\zeronorm{\vbeta^*}, 1 + \zeronorm{\vbeta^*}\}$, where $\zeronorm{\vbeta^*}$ denotes the number of non-zero components of $\vbeta^*$.  The value $k(\vmu)$ is small when
$\vbeta^*$ is sparse.  Thus, the problem of estimation of piecewise constant sequence $\vmu$ with small number of pieces can  be considered as the problem of prediction in sparse linear regression with a specific design matrix $\design$. Similarly, we may write $\vu=\design\vbeta$,
for $\vbeta$ with components $\beta_1 = u_1$ and $\beta_i = u_i - u_{i-1}$ for $i=2,\dots,n$. These remarks suggest that we can apply the theory of sparsity oracle inequalities, in particular, sparsity pattern aggregation (cf. \cite{RigolletTsybakov2011,rigollet2012sparse,tsybakovICM}) in the context of monotone estimation described above. Similar observation is valid for estimation under convexity constraints (see Section \ref{sec:convex} below). In the present paper, we develop this argument using as a building block the $Q$-aggregation procedures \cite{rigollet2012kullback,dai2012deviation,dai2014aggregation,bellec2014affine}.
In particular, we  construct an estimator $\hmu$ such that
    \begin{equation}
        \E\scalednorms{\hmu -\vmu}
        \le
        \min_{\vu\in\increasings}
        \left(
        \scalednorms{\vmu - \vu} + \frac{c \sigma^2k(\vu)}{n} \log \frac{en}{k(\vu)}
        \right), \quad \forall \ \vmu\in\Rn,
        \label{eq:rhs-beta} 
    \end{equation}
    for some absolute constant $c>0$. Note that \eqref{eq:rhs-beta} is a sharp oracle inequality (i.e., an inequality with leading constant 1).
 It   improves upon the oracle inequality \eqref{eq:lse-oracle-ineq} for the least squares estimator where the leading constant $C_*$ is noticeably greater than~1 and the bound is valid only for $\vmu\in\increasings$. 
The advantage of having leading constant~1 and arbitrary $\vmu$ in 
\eqref{eq:rhs-beta} is that it allows us to derive bounds on the excess risk and on the minimax regret, which was not possible based on the previous results. We also obtain sharp oracle inequalities with high probability for the same estimator.  In addition, we show that it satisfies stronger sharp inequalities with the minimum $\min_{\vu\in\increasings}$ on the right hand side of \eqref{eq:rhs-beta} replaced by $\min_{\vu\in\Rn}$.  
This implies that our results are invariant to the direction of monotonicity; they remain valid if we replace everywhere monotone increasing by monotone decreasing functions.
Finally, we derive similar results for the problem of estimation under the  convexity constraints improving an oracle inequality obtained in~\cite{guntuboyina2013global}.

\section{Sparsity pattern aggregation for piecewise constant sequences} \label{s:spa}

For any non-empty set $J\subseteq\{1,...,n-1\}$, let $|J|$ denote the cardinality of $J$ and define
\begin{equation}
    \pi_J \coloneqq \frac{\exp(-|J|)}{H \; {n-1 \choose |J|}},
        \qquad
        H \coloneqq \sum_{i=0}^{n-1} \exp(-i).
        \label{eq:prior-monotone}
\end{equation}
Let $P_J\in \R^{n\times n}$ be the projector on the linear subspace $V_J$ of $\Rn$ defined by
\begin{equation}
    V_J \coloneqq \Big\{\vu\in\Rn: \; \forall i \in \{1,...,n-1\} \setminus J,\; u_{i+1} = u_i \Big\}.
\end{equation}
In words, $V_J$ is the space of all piecewise constant sequences that have jumps only at points in $J$. Given a vector $\vy$ of observations and $\vtheta = (\theta_J)_{J\subseteq \{1,...,n-1\}}$ where each $\theta_J\in \R$, let
\begin{equation}
    \vmu_{\vtheta} = \sum_{J\subseteq\{ 1,...,n-1\}} \theta_J P_J \vy.
\end{equation}
Finally, let 
\begin{equation}
\label{eq:def-q}
    \hmu^{Q} = \vmu_{\htheta}
\end{equation} 
where $\htheta$ is the solution of the optimization problem
\begin{align*}
    \min_{\vtheta\in \Lambda} \hspace{1em}& 
    \scalednorms{\vmu_\vtheta - \vy} 
    + \sum_{J\subseteq\{1,...,n-1\}} \theta_J 
    \left(
        \frac{2 \sigma^2 |J|}{n} + \frac{1}{2} \scalednorms{\vmu_\vtheta - P_J\vy}
        + \frac{46 \sigma^2}{n} \log\frac{1}{\pi_J} 
    \right)
    \end{align*}
    where
    $$
    \Lambda=\left\{\vtheta : \
     \theta_J \ge 0   \text{ for all } J\subseteq\{1,...,n-1\}, \text{ and }
              \sum_{J\subseteq\{1,...,n-1\}} \theta_J = 1
\right\}.
    $$
This optimization problem is a convex quadratic program with a simplex constraint.
It performs aggregation of the linear estimators $(P_J \vy)_{J\subseteq\{1,...,n-1\}}$
using the $Q$-aggregation procedure \cite{dai2012deviation,dai2014aggregation,bellec2014affine} with the prior weights \eqref{eq:prior-monotone}.
As the size of this quadratic program is of order $2^n$,
it is a computationally hard problem.
The estimator $\hmu^Q$ satisfies the following sharp oracle inequalities.

\begin{theorem}
    \label{thm:q-monotone}
    Let $\vmu\in\Rn$, $n\ge2$, and assume that the noise vector $\vxi$ has distribution $\mathcal{N}(0, \sigma^2 I_{n\times n})$.
    There exist absolute constants $c,c'>0$ such that for all $\delta\in(0,1/3)$, the estimator $\hmu^Q$ 
    satisfies with probability at least $1-3\delta$,
    \begin{equation}
        \scalednorms{\hmu^{Q} - \vmu} \le \min_{\vu\in\Rn}
        \left(
            \scalednorms{\vmu - \vu} +
            \frac{ c \sigma^2 k(\vu))}{n} 
            \log \frac{en}{ k(\vu)}
        \right)
            + \frac{c\sigma^2 \log(1/\delta)}{n},
        \label{eq:soi-q-deviation}
    \end{equation}
    and
    \begin{equation}
        \E \scalednorms{\hmu^{Q} - \vmu} \le \min_{\vu\in\Rn}
        \left(
            \scalednorms{\vmu - \vu} + 
            \frac{ c' \sigma^2k(\vu)}{n} 
            \log \frac{en}{k(\vu)}
        \right).
        \label{eq:soi-q}
    \end{equation}
\end{theorem}
\begin{proof}
    Let $J\subseteq\{1,...,n-1\}$. Denote by $d=|J|+1$ the dimension of the subspace $V_J$.
    Then, the projection estimator $P_J\vy$ satisfies
    with probability at least $1-\delta$ (see, for example, \cite{hsu2012tail}):
    \begin{align}\nonumber
        \scalednorms{P_J\vy - \vmu} & \le \scalednorms{P_J\vmu - \vmu} + \frac{d + 2\sqrt{d\log(1/\delta)} + 2 \log(1/\delta)}{n} \\
                                    & \le \min_{\vu\in V_J}\scalednorms{\vu - \vmu} + \frac{2(|J|+1) + 3 \log(1/\delta)}{n}.
    \label{eq:th21_1}
    \end{align}
    The sharp oracle inequality from \cite{bellec2014affine} yields
    that with probability at least $1-2\delta$ for all $J\subseteq\{1,...,n-1\}$ we have
    \begin{align}  \label{eq:th21_2}
        \scalednorms{\hmu^Q - \vmu} & \le \scalednorms{P_J\vy - \vmu} + C\sigma^2 \log\frac{1}{\pi_J} + C \sigma^2 \log(1/\delta),
    \end{align}
    for some absolute constant $C>0$.
    Combining   \eqref{eq:th21_1} and   \eqref{eq:th21_2} with the union bound and the inequality  (cf. \cite[(5.4)]{rigollet2012sparse})
    $\log(1/\pi_J) \le 2(|J|+1) \log(en/(|J|+1)) + 1/2$, we find that with probability at least $1-3\delta$,
   \begin{align*}  
        \scalednorms{\hmu^Q - \vmu} & \le \min_{J\subseteq\{1,...,n-1\}}\min_{\vu\in V_J}\left(\scalednorms{\vmu-\vu} + \frac{c\sigma^2(|J|+1)}{n} \log\left(\frac{en}{|J|+1}\right) \right) 
        \\
        &\quad \ \ + c \sigma^2 \log(1/\delta)
    \end{align*}
 where $c>0$ is an absolute constant.  Since  $|J|+1 = k(\vu)$ for all $\vu\in V_J$ and $\min_{J\subseteq\{1,...,n-1\}} \min_{\vu\in V_J}=\min_{\vu\in \Rn}$, the bound \eqref{eq:soi-q-deviation} follows. 
Finally, \eqref{eq:soi-q} is obtained from \eqref{eq:soi-q-deviation} by integration.
\end{proof}

We now discuss some corollaries of Theorem~\ref{thm:q-monotone}. First, it follows that \eqref{eq:rhs-beta} is satisfied for $\hmu=\hmu^Q$, so the remarks after \eqref{eq:rhs-beta} apply. Next, in view of \eqref{eq:soi-q}, for the class of monotone sequences with at most $k$ jumps $\increasings_k = \{\vu\in \increasings: k(\vu)\le k\}$ we have the following bounds for the maximal expected regrets 
\begin{align}  
       \max_{\vmu\in\Rn} \left(\E\|\hmu^Q - \vmu\| - \min_{\vu\in \increasings_k}\Vert\vu - \vmu\Vert\right) & \le c\sqrt{\frac{\sigma^2k}{n} \log\left(\frac{en}{k}\right)}, \label{eq:max_expected_regret}
       \\
       \max_{\vmu\in\Rn} \left(\E\|\hmu^Q - \vmu\|^2 - \min_{\vu\in \increasings_k}\Vert\vu - \vmu\Vert^2 \right) &\le \frac{c\sigma^2k}{n} \log\left(\frac{en}{k}\right), \label{eq:max_expected_regret2}
    \end{align}
where $c>0$ is an absolute constant.
The same bounds hold for the minimax risks over $\increasings_k$ since the minimax risk is smaller than  the minimax regret.
\Cref{prop:lower} below shows that the bounds \eqref{eq:max_expected_regret} and \eqref{eq:max_expected_regret2} are optimal
up to logarithmic factors.

Finally, consider the consequences of Theorem~\ref{thm:q-monotone} for the class $\increasings(V)$. 
To this end, define the integer  $k^*$  such that
    $$
    k^* = \min\left\{ m\in \mathbb N: m \ge \left(\frac{V(\vmu)^2 n}{\sigma^2\log(en)}\right)^{1/3}\right\}
    $$
if the set $\left\{ m\in \mathbb N: m \ge \left(\frac{V(\vmu)^2 n}{\sigma^2\log(en)}\right)^{1/3}\right\}$ is non-empty, and $k^*=1$ otherwise.
We will need the following lemma.

\begin{lemma}
    \label{prop:approximation}
    Let $\vmu\in\increasings$ and let $1\le k\le n$ be an integer. Then
    there exists a sequence $\bar\vu\in\increasings_k$
    such that
    \begin{equation}
    \|\bar\vu - \vmu\| \le \frac{V(\vmu)}{2k}. \label{eq:approximation0}
    \end{equation}
    Next,     there exists a sequence $\bar\vu\in\increasings_{k^*}$
    such that
    \begin{equation}
        \scalednorms{\bar\vu - \vmu}
        \le
        \frac{1}{4}
        \max\left( 
            \left(\frac{\sigma^2 V(\vmu) \log(en)}{n}\right)^{2/3}
            ,
            \frac{\sigma^2 \log(en)}{n}
        \right)
        .
        \label{eq:approximation}
    \end{equation}
    In addition, 
    \begin{equation}
        \frac{\sigma^2 k^*}{n} \log\frac{en}{k^*}
        \le
        2 \max\left( 
            \left(\frac{\sigma^2 V(\vmu) \log(en)}{n}\right)^{2/3}
            ,
            \frac{\sigma^2 \log(en)}{n}
        \right)
        .
        \label{eq:simple-bound}
    \end{equation}
\end{lemma}
\begin{proof}
    To construct the sequence $\bar\vu$,
    consider the $k$ intervals
    \begin{equation}
        I_j= \Big[
            \mu_1 + \frac{j-1}{k} V(\vmu),
            \mu_1 + \frac{j}{k} V(\vmu)
            \Big),
            \qquad
            j=1,...,k-1,
    \end{equation}
    and $I_k = [\mu_1 + \frac{k-1}{k} V(\vmu),\mu_n]$.
    For all $j=1,...,k$, let
    \begin{equation}
        J_j = \{i=1,...,n: \quad \mu_i\in I_j\}.
    \end{equation}
    For any $i\in \{1,...,n\}$ there exists a unique $j\in\{1,...,k\}$
    such that $i\in I_j$.
    Let $\bar u_i = \mu_1 + \frac{j-1/2}{k} V(\vmu)$ for all $i\in I_j$.
    Then the sequence $\bar\vu=(\bar u_1,\dots,\bar u_n)$ is non-decreasing, it has at most $k$ pieces, i.e., $k(\bar\vu)\le k$, and  $|\bar  u_i - \mu_i| \le \frac{V(\vmu)}{2k}$ for  $i=1,...,n$. 
    Thus \eqref{eq:approximation0} follows.
    Next, note that if $k^*=1$, then $V(\vmu)^2 \le \sigma^2 \log(en)/n$.
    If $k^*>1$, then by definition of $k^*$, $V(\vmu)^2 /(k^*)^2 \le (\sigma^2 V(\vmu) \log(en)/n)^{2/3}$. Thus, \eqref{eq:approximation}
    follows.
    The bound \eqref{eq:simple-bound} is straightforward
    by studying the cases $k^*=1$ and $k^*>1$ separately.
\end{proof}

We can now derive the following corollary of Theorem~\ref{thm:q-monotone}.

\begin{corollary}\label{cor:mon}
Under the assumptions of Theorem~\ref{thm:q-monotone}, there exists an absolute constant $c>0$ such that, for any $\vmu\in \increasings$,
\begin{equation}\label{eq:cor:mon1}
\E\|\hmu^Q - \vmu\|^2 \le 
c \,\max\left( 
            \left(\frac{\sigma^2 V(\vmu) \log n}{n}\right)^{2/3}
            ,
            \frac{\sigma^2 \log n}{n}
        \right)
        .
\end{equation}
In addition, for any $V>0$ and  any  $\vmu\in \Rn$ the expected regret  of  $\hmu^Q$ satisfies 
\begin{equation}\label{eq:cor:mon2}
 \E\|\hmu^Q - \vmu\| - \min_{\vu\in \increasings(V)}\Vert\vu - \vmu\Vert
\ \le 
c \,\max\left( 
            \left(\frac{\sigma^2 V \log n}{n}\right)^{1/3}
            ,
            \sigma \sqrt{\frac{\log n}{n}}
        \right)
\end{equation}
where $c>0$ is an absolute constant.
\end{corollary}
\begin{proof} Inequality \eqref{eq:cor:mon1} is straightforward in view of \eqref{eq:soi-q}, \eqref{eq:approximation}, and \eqref{eq:simple-bound}.
To prove \eqref{eq:cor:mon2}, fix any $\vmu\in \Rn$ and consider
$$
\vmu^* \in \argmin_{\vmu'\in \increasings(V)}\|\vmu' - \vmu\|.
$$
From  \eqref{eq:soi-q} and the fact that the function $x\mapsto x\log\left(\frac{en}{x}\right)$ is increasing for $1\le x\le n$ we get
\begin{align*}  
      \E\|\hmu^Q - \vmu\| &\le   \min_{\vu\in \increasings_{k^*}}\left(\Vert\vu - \vmu\Vert + \sqrt{c'\frac{\sigma^2k^*}{n} \log\left(\frac{en}{k^*}\right)} \right)
      \\
      &\le \min_{\vu\in \increasings_{k^*}}\Vert\vu - \vmu^*\Vert + \Vert\vmu^* - \vmu\Vert + \sqrt{c'\frac{\sigma^2k^*}{n} \log\left(\frac{en}{k^*}\right)} 
      \\
      &\le \Vert\vmu^* - \vmu\Vert + 
c'' \,\max\left( 
            \left(\frac{\sigma^2 V \log n}{n}\right)^{1/3}
            ,
            \sigma \sqrt{\frac{\log n}{n}}
        \right)
        \end{align*}
        for an absolute constant $c''>0$ where the last inequality follows from \eqref{eq:approximation} and \eqref{eq:simple-bound}.
\end{proof}

The estimator $\hmu^Q$ shown in \Cref{thm:q-monotone}
satisfies the sharp oracle inequalities both in expectation and with high probability.
Previous results for the least squares estimator \cite{chatterjee2013risk}
were only obtained in expectation
and the results on the $\ell_1$-penalized estimator \eqref{eq:def-tv}
are only obtained with high probability.

Finally, the following result shows
that the upper bounds 
\eqref{eq:max_expected_regret}
and \eqref{eq:max_expected_regret2}
are optimal up to logarithmic factors.

\begin{proposition}
    \label{prop:lower}
   Let $n\ge 2, V>0 \text{ and }  \sigma>0$.~There exist absolute constants $c,c'>0$ such that for any positive integer $k\le n$
    satisfying $k^3 \le 16 n V^2 / \sigma^2$ we have
    \begin{equation}
        \inf_{\hmu}
        \sup_{\vmu\in\increasings_k \cap \increasings(V)}
        \mathbb{P}_{\vmu}
        \left(
            \|\hmu - \vmu\|^2 \ge \frac{c \sigma^2 k}{n}
        \right) > c',
        \label{eq:lower-k-over-n}
    \end{equation}
    where $ \mathbb{P}_{\vmu}$ denotes the distribution of $\vy$ satisfying \eqref{eq:observations} and $\inf_{\hmu}$ is the infimum over all estimators.
\end{proposition}
For $k=1,...,n$, take any $V>0$ large enough to satisfy
$k^3 \le 16 n V^2 / \sigma^2$. Then, \Cref{prop:lower} and 
Markov's inequality
 yield the following lower bounds on the minimax risks over the class $\increasings_k$:
\begin{equation}\label{lower_minimax}
\inf_{\hmu}
    \sup_{\vmu\in\increasings_k}
    \E \|\hmu - \vmu\| \ge c \sqrt{\frac{c' \sigma^2 k}{n}},
\quad
    \inf_{\hmu}
    \sup_{\vmu\in\increasings_k}
    \E \scalednorms{\hmu - \vmu} \ge \frac{c c' \sigma^2 k}{n}.
\end{equation}
As the minimax risk is smaller than the minimax regret, \eqref{lower_minimax} also provides lower bounds for the corresponding minimax regrets over $\increasings_k$.
Combining this with 
\eqref{eq:max_expected_regret} and \eqref{eq:max_expected_regret2}
we find that the estimator $\hmu^Q$ achieves up to logarithmic factors the optimal rate with respect to the minimax regret.

Next, Proposition \ref{prop:lower}  implies the following lower bound on the minimax deviation risk on $\increasings(V)$.
\begin{corollary}\label{cor:lower}
Let $n\ge 2, V>0 \text{ and }  \sigma>0$.~There exist absolute constants $c,c'>0$ such that     
\begin{equation}
        \inf_{\hmu}
        \sup_{\vmu\in\increasings(V)}
        \mathbb{P}_{\vmu}
        \left(
            \|\hmu - \vmu\|^2 \ge c \max\left\{\left(\frac{ \sigma^2 V}{n}\right)^{2/3}, \frac{\sigma^2}{n}
       \right\} \right) > c'.
       \label{eq:lower-n23}
           \end{equation}
\end{corollary}

To prove this corollary it is enough to  note that if $16 n V^2 / \sigma^2 \ge 1$, by choosing $k$ in Proposition \ref{prop:lower} as the integer part of $(16 n V^2 / \sigma^2)^{1/3}$, we obtain the lower bound corresponding to $\left(\frac{ \sigma^2 V}{n}\right)^{2/3}$ under the maximum in \eqref{eq:lower-n23}. On the other hand, if  $16 n V^2 / \sigma^2 < 1$ the term $\frac{\sigma^2}{n}$ is dominant, so that we need to have the lower bound of the order $\frac{\sigma^2}{n}$, which is trivial (it follows from a reduction to the bound for the class composed of two constant functions).

It follows from \eqref{eq:lower-n23}
and \eqref{eq:cor:mon1}  that 
the estimator $\hmu^Q$ 
 achieves, up to logarithmic factors,  the optimal rate with respect to the minimax risk
on the class $\increasings(V)$.
Using \eqref{eq:cor:mon2} and the fact that the minimax risk is smaller than the minimax regret,
we conclude that it is also the optimal rate up to logarithmic factors for the minimax regret.

\begin{proof}[Proof of \Cref{prop:lower}]
    We assume for simplicity that $n$ is a multiple of $k$. The general case is treated analogously.
    For any $\vomega,\vomega'\in\{0,1\}^k$, 
    let $d_H(\vomega,\vomega') = |\{i=1,...,k: \omega_i \ne \omega_i' \}|$
    be the Hamming distance between $\vomega$ and $\vomega'$.
    By the Varshamov-Gilbert bound \cite[Lemma 2.9]{tsybakov2009introduction},
    there exists a set $\Omega \subset \{0,1\}^k$ such that 
    \begin{equation}
        {\bf 0} = (0,...,0)\in\Omega,
        \quad
        \log (|\Omega| - 1)\ge k/8,
        \quad
        \text{and}
        \quad
        d_H(\vomega,\vomega') > k/8  
        \label{eq:varshamov}
    \end{equation}
    for any two distinct $\vomega, \vomega' \in\Omega$.
       For each $\vomega\in\Omega$,
    define a vector $\vu^\vomega\in\Rn$  with components
    \begin{equation*}
        u^\vomega_i = \frac{\lfloor (i-1)k/n \rfloor \; V}{2k} + \gamma \omega_{\lfloor (i-1)k/n \rfloor + 1},
        \qquad
        i=1,...,n,
    \end{equation*}
    where 
    $\gamma = (1/8) \sqrt{\sigma^2 k / n}$, and $\lfloor x\rfloor$ denotes the maximal integer smaller than $x$.
    For any $\vomega\in\Omega$, $\vu^\vomega$ is a piecewise constant sequence with $k(\vu^\vomega)\le k$,
    $\vu^\vomega$ is a non-decreasing sequence because $\gamma \le V/(2k)$,
    and by construction $V(\vu^\vomega)\le V$. Thus, $\vu^\vomega \in \increasings_k\cap\increasings(V)$ for all  $\vomega\in\Omega$.
    Moreover, for any $\vomega,\vomega'\in\Omega$,
    \begin{equation}\label{varsham}
        \|\vu^\omega - \vu^{\vomega'}\|^2
        =
        \frac{\gamma^2}{k}
        d_H(\vomega,\vomega')
        \ge
        \frac{\gamma^2}{8} = \frac{\sigma^2 k}{512n}.
    \end{equation}
    Set for brevity $P_\vomega={\mathbb P}_{\vu^\omega}$.     The Kullback-Leibler divergence $K(P_\vomega,P_{\vomega'})$ between $P_\vomega$ and $P_{\vomega'}$
    is equal to $\frac{n}{2\sigma^2}  \|\vu^\omega - \vu^{\vomega'}\|^2$ for all $\vomega,\vomega'\in\Omega$.
    Thus,
    \begin{equation}\label{Kullback}
        K(P_\vomega,P_{\bf 0})
        =
        \frac{\gamma^2 n d_H({\bf 0}, \vomega)}{2 k \sigma^2}
        \le \frac{k}{128}
        \le \frac{\log(|\Omega| - 1)}{16}.
    \end{equation}
    Applying \cite[Theorem 2.7]{tsybakov2009introduction} with $\alpha = 1/16$ completes the proof.
\end{proof}

\section{Estimation of convex sequences by aggregation}\label{sec:convex}

Assume that $n\ge 3$ and define the set of convex sequences $\convexs$ as follows:
\begin{equation}
    \convexs = \{ \vu=(u_1,\dots,u_n)\in\Rn: \; 2u_i \le u_{i+1} + u_{i-1}, \  i=2,\dots,n-1 \}.
\end{equation}
For any $\vu\in\Rn$,
we introduce the integer $q(\vu)\ge 1$ such that $q(\vu)- 1$ is the cardinality of the set
$\{i=1,...,n-1: 2u_i \ne u_{i+1} + u_{i-1} \}$.
If $\vu\in\convexs$, $q(\vu)-1$ is the number of inequalities $2u_i \le u_{i+1} + u_{i-1}$ that are strict for $i=2,...,n-1$.
The value $q(\vu)$ is small if $\vu$ is a piecewise linear sequence with a small number of pieces.

The performance of the least squares estimator over convex sequences $\hmu^{LS}(\convexs)$ has been recently studied in \cite{guntuboyina2013global}.
If the unknown vector $\vmu$ belongs to the set $\convexs$, \cite{guntuboyina2013global} shows that
the estimator $\hmu^{LS}(\convexs)$
satisfies the risk bound
\begin{equation*}
    \E\scalednorms{\hmu^{LS}(\convexs) - \vmu}
    \le
    c
    \log(en)^{5/4}
    \left(\frac{\sigma^2 \sqrt{R(\vmu)}}{n}\right)^{4/5},
\end{equation*}
where $R(\vmu) = \max(1, \min \{\scalednorms{\boldsymbol{\tau} - \vmu}, \boldsymbol{\tau} \text{ is affine}\})$ and $c>0$ is an absolute constant.
It is proved in \cite[Example 2.3]{chatterjee2013risk} that the least squares estimator $\hmu^{LS}(\convexs)$ satisfies the oracle inequality
\begin{equation}
    \E\scalednorms{\hmu^{LS}(\convexs) - \vmu}
    \le
    6
    \min_{\vu\in\convexs}
    \left(
        \scalednorms{\vu-\vmu} + \frac{c\sigma^2 q(\vu) \log\left(\frac{en}{q(\vu)}\right)^{5/4}}{n}
    \right),
    \label{eq:oi-lse-convex}
\end{equation}
where $c>0$ is an absolute constant.
The right hand side of \eqref{eq:oi-lse-convex} 
is small if the unknown vector $\vmu$ can be well approximated
by a piecewise linear sequence in $\convexs$ with not too many pieces.

The leading constant in \eqref{eq:oi-lse-convex} is 6.
We will show that sparsity pattern aggregation achieves a substantially better performance.
We obtain the sharp oracle inequality \eqref{eq:soi-q-convex} below,
improving upon \eqref{eq:oi-lse-convex} not only in the fact that the leading constant is 1 but also in the rate of the
remainder term; we will see that the exponent $5/4$ of the logarithmic factor is reduced to 1.

For any set $J\subseteq\{2,...,n-1\}$, define
\begin{equation}
    \nu_J \coloneqq \frac{\exp(-|J|)}{H_C \; {n-2 \choose |J|}},
        \qquad
        H_C \coloneqq \sum_{i=0}^{n-2} \exp(-i).
    \label{eq:prior-convex}
\end{equation}
Let $Q_J\in \R^{n\times n}$ be the projector on the linear subspace $W_J$ of $\Rn$ given by
\begin{equation*}
    W_J \coloneqq \Big\{\vu\in\Rn: \; \forall i \in \{2,...,n-1\} \setminus J,\; 2u_i = u_{i+1} + u_{i-1} \Big\}.
\end{equation*}
Given a vector $\vy$ of observations and $\vtheta = (\theta_J)_{J\subseteq \{2,...,n-1\}}$ where each $\theta_J$ belongs to $\R$, let
\begin{equation*}
    \vmu_{\vtheta} = \sum_{J\subseteq\{ 2,...,n-1\}} \theta_J Q_J \vy.
\end{equation*}
Finally, let 
\begin{equation*}
\label{eq:def-q-convex}
    \hmu^{Q-conv} = \vmu_{\htheta}
\end{equation*} 
where $\htheta$ is the solution of the optimization problem
\begin{align*}
    \min_{\vtheta\in  \Lambda'} \ \
    \scalednorms{\vmu_\vtheta - \vy} 
    + \sum_{J\subset\{2,...,n-1\}} \theta_J 
    \left(
        \frac{2 \sigma^2 |J|}{n} + \frac{1}{2} \scalednorms{\vmu_\vtheta - Q_J\vy}
        + \frac{46 \sigma^2}{n} \log\frac{1}{\nu_J} 
    \right)
    \end{align*}
 where
    $$
    \Lambda'=\left\{\vtheta : \
     \theta_J \ge 0   \text{ for all } J\subseteq\{2,...,n-1\}, \text{ and }
              \sum_{J\subseteq\{2,...,n-1\}} \theta_J = 1
\right\}.
    $$
The structure of this minimization problem is the same as of its analog introduced in Section \ref{s:spa}. 
This is a quadratic program
that aggregates the linear estimators $(Q_J \vy)_{J\subseteq\{2,...,n-1\}}$
using the $Q$-aggregation procedure \cite{dai2012deviation,dai2014aggregation,bellec2014affine} with the prior weights \eqref{eq:prior-convex}.

\begin{theorem}
    \label{thm:q-convex}
    Let $\vmu\in\Rn$, $n\ge 3$, and assume that the noise vector $\vxi$ has distribution $\mathcal{N}(0, \sigma^2 I_{n\times n})$.
    There exist absolute constants $c,c'>0$ such that for all $\delta\in(0,1/3)$, the estimator $\hmu^{Q-conv}$ 
    satisfies with probability at least $1-3\delta$,
    \begin{equation} \label{eq:soi-q-convex0}
        \scalednorms{\hmu^{Q-conv} - \vmu} \le \min_{\vu\in\Rn}
        \left(
            \scalednorms{\vmu - \vu} +
            \frac{ c \sigma^2q(\vu)}{n} 
            \log \frac{en}{ q(\vu)}
        \right)
            + \frac{c\sigma^2 \log(1/\delta)}{n},
    \end{equation}
    and we have
    \begin{equation}
    \E \scalednorms{\hmu^{Q-conv} - \vmu} \le \min_{\vu\in\Rn}
        \left(
            \scalednorms{\vmu - \vu} + 
            \frac{ c' \sigma^2q(\vu)}{n} 
            \log \frac{en}{q(\vu)}
        \right).
        \label{eq:soi-q-convex}
    \end{equation}
\end{theorem}
The proof of this theorem is the same as that of \Cref{thm:q-monotone} with the only difference that $J$ is now a subset of $\{2,...,n-1\}$ rather than that of $\{1,...,n-1\}$, and we replace the notation $P_J $ and $V_J$ by $Q_J $ and $W_J$ respectively.

The leading constant of the oracle inequality \eqref{eq:soi-q-convex} is 1,
and the remainder term is proportional to $q(\vu)\log(en/q(\vu))$.
These are two improvements upon \eqref{eq:oi-lse-convex},
where the leading constant is 6
and the remainder term is proportional to $q(\vu)\log(en/q(\vu))^{5/4}$. 

In view of \eqref{eq:soi-q-convex}, for the class of piecewise linear convex sequences with at most $q$ linear pieces, $\convexs_q = \{\vu\in \convexs: q(\vu)\le q\}$ we have the following bounds for the maximal expected regrets 
\begin{align}  
       \max_{\vmu\in\Rn} \left(\E\|\hmu^Q - \vmu\| - \min_{\vu\in \convexs_q}\Vert\vu - \vmu\Vert\right) & \le c\sqrt{\frac{\sigma^2q}{n} \log\left(\frac{en}{q}\right)}, \label{eq:max_expected_regret3}
       \\
       \max_{\vmu\in\Rn} \left(\E\|\hmu^Q - \vmu\|^2 - \min_{\vu\in \convexs_q}\Vert\vu - \vmu\Vert^2 \right) &\le \frac{c\sigma^2q}{n} \log\left(\frac{en}{q}\right), \label{eq:max_expected_regret4}
    \end{align}
where $c>0$ is an absolute constant.
The same bounds hold for the minimax risks over $\convexs_q$ since the minimax risk is smaller than  the minimax regret.

The following proposition shows
that the rates of convergence in \eqref{eq:max_expected_regret3} and \eqref{eq:max_expected_regret4} are optimal up to logarithmic factors.
We omit the discussion since it is similar to that after \Cref{prop:lower}.

\begin{proposition}
    \label{prop:lower-convex}
    Let $n\ge 3$. There exist absolute constants $c,c'>0$ such that, for any positive integer $q\le n$,
    \begin{equation}
        \inf_{\hmu}
        \sup_{\vmu\in\convexs_q}
        \mathbb{P}_{\vmu}
        \left(
           \|{\hmu - \vmu}\|^2 \ge \frac{c \sigma^2 q}{n}
        \right) > c'
        ,
        \label{eq:lower-q-over-n-convex}
    \end{equation}
    where the infimum is taken over all estimators.
\end{proposition}
\begin{proof} Assume that $q\ge 2$ since for $q=1$ the result is trivial.
    We also assume for simplicity that $n$ is a multiple of $q$.
    Let $m=n/q$
    and
    $\gamma = (1/8) \sqrt{\sigma^2 q / n}$.
    Set $\beta_0 = 0, \alpha_0 = 0$ and define, for all integers $j\ge 1$,
    \begin{equation}
        \beta_j =\beta_{j-1} + \gamma + m \alpha_{j-1},
        \qquad
        \alpha_j = 2 \gamma + \alpha_{j-1}.
        \label{eq:def-alpha-beta}
    \end{equation}
    By the Varshamov-Gilbert bound \cite[Lemma 2.9]{tsybakov2009introduction} there exists
    $\Omega\subset \{0,1\}^q$ such that \eqref{eq:varshamov} is satisfied, with $k$ replaced by $q$.
    For each $\vomega\in\Omega$,
    define a vector $\vu^\vomega\in\Rn$  with components
    \begin{equation*}
        u^\vomega_{jm+i} = \omega_{j+1} \gamma + \alpha_j (i-1) + \beta_j,
        \qquad
        j=0,...,q-1,
        \quad
        i=1,...,m.
    \end{equation*}
 The sequence $\vu^\vomega$ is piecewise linear.
   It is linear with slope $\alpha_j$ on the set $\{jm+1,...,(j+1)m \}$ for any $j=0,...,q-1$. Thus, $q(\vu^\vomega)=q$.
    Next, we prove that $\vu^\vomega\in \convexs$ for all $\vomega\in \Omega$. It is enough to check the convexity condition at the endpoints of the linear pieces:
     \begin{align}
         \label{12}
         2u^\vomega_{jm}\le u^\vomega_{jm-1} + u^\vomega_{jm+1},
         \qquad
         2u^\vomega_{jm+1}\le u^\vomega_{jm} + u^\vomega_{jm+2},
     \end{align}
     for all $j=1,...,q-1$.
     Using
    \eqref{eq:def-alpha-beta} we get
    that, for all $j=1,...,q-1$,
    \begin{align*}
        u^\vomega_{jm+1}-u^\vomega_{jm}
       &= \omega_{j+1}\gamma + \beta_j  - (\omega_{j}\gamma   + \alpha_{j-1}(m-1) + \beta_{j-1}),
        \\
        &= (\omega_{j+1} - \omega_{j}+1)\gamma + \alpha_{j-1}, \\
        &= (\omega_{j+1} - \omega_{j} - 1)\gamma + \alpha_{j}.
    \end{align*}
    Hence, $\alpha_{j-1} \le u^\vomega_{jm+1}-u^\vomega_{jm} \le \alpha_j$.
    Since also $\alpha_{j-1} = u^\vomega_{jm}-u^\vomega_{jm-1}$
    and $\alpha_j = u^\vomega_{jm+2}-u^\vomega_{jm+1}$,
    it follows that the two inequalities \eqref{12} hold, for all $j=1,...,q-1$. Thus, 
    $\vu^\vomega\in\convexs$. In summary, we have proved that $\vu^\vomega\in \convexs_q$ for all $\vomega\in \Omega$.   
    
    Now, from the Varshamov-Gilbert bound, cf. \eqref{eq:varshamov},  for $\vomega,\vomega'\in\Omega$ we have
    \begin{equation}
        \|\vu^\omega - \vu^{\vomega'}\|^2
        =
        \frac{\gamma^2}{q}
        d_H(\vomega,\vomega')
        \ge
        \frac{\gamma^2}{8} = \frac{\sigma^2 q}{512n},
    \end{equation}
    where $d_H(\cdot,\cdot)$ is the Hamming distance.
    Finally, similarly to \eqref{Kullback}, the Kullback-Leibler divergence between $P_\vomega$  and $P_{\bf 0}$ satisfies
        $
        K(P_\vomega,P_{\bf 0})
                \le \frac{\log(|\Omega| - 1)}{16}$.
    Applying \cite[Theorem 2.7]{tsybakov2009introduction} with $\alpha = 1/16$ completes the proof.
\end{proof}

\section{Concluding remarks and discussion}
In this short note,
we have shown that the estimators
$\hmu^{Q}$ and $\hmu^{Q-conv}$
based on 
sparsity pattern aggregation (in its $Q$-aggregation version)
achieve oracle inequalities
that improve on some previous results for isotonic and convex regression.

One of the improvements is that oracle inequalities 
\eqref{eq:soi-q} and \eqref{eq:soi-q-convex} are sharp, i.e., 
with leading constant  1 and they are valid for all $\vmu\in\Rn$.
It allows us to obtain bounds for the minimax regret under arbitrary model misspecification, which was not possible based on the previous results. We show that these bounds are rate optimal up to logarithmic factors.
The question on whether the least squares estimators under monotonicity and convexity constraints can achieve sharp oracle inequalities 
with correct rates remains open. 

Another improvement is that we obtain oracle inequalities both with high probability and in expectation, which was not the case in the previous work. 

An advantage of the least squares estimator is that it requires
no tuning parameters.
In particular, the knowledge of $\sigma^2$
is not needed to construct
the estimators $\hmu^{LS}(\increasings)$ and $\hmu^{LS}(\convexs)$.
This is in contrast to the $\ell_1$ penalized estimator \eqref{eq:def-tv} and the estimators $\hmu^{Q}$ and $\hmu^{Q-conv}$;
their construction requires the knowledge of $\sigma^2$.
For the $\ell_1$ penalized estimator \eqref{eq:def-tv},
the issue may be addressed by using a scale-free 
version of the Lasso \cite{belloni2014pivotal,sun2012scaled}.
For the $Q$-aggregation estimators $\hmu^{Q}$ and $\hmu^{Q-conv}$, we can treat the issue of unknown $\sigma$ as in \cite{bellec2014affine}. Namely, it is shown in \cite{bellec2014affine}
that the oracle inequalities for $Q$-aggregation procedures are essentially preserved after plugging in an estimator $\hat\sigma^2$ of $\sigma^2$ that
satisfies $|\hat\sigma^2/\sigma^2 - 1| \le 1/8$ with high probability, which 
is even weaker than consistency.

Finally, note that instead of $Q$-aggregation we could have used sparsity pattern aggregation by the Exponential Screening procedure of \cite{RigolletTsybakov2011}. This would lead to sharp oracle inequalities in expectation of the form \eqref{eq:soi-q} and \eqref{eq:soi-q-convex}
but not to inequalities with high probability such as \eqref{eq:soi-q-deviation} and \eqref{eq:soi-q-convex0}. This is the reason why we have opted for $Q$-aggregation rather than for Exponential Screening in this paper. On the other hand, Exponential Screening estimators are computationally more attractive than $Q$-aggregation since they can be successfully approximated by MCMC algorithms (see \cite{RigolletTsybakov2011,rigollet2012sparse} for details).

\smallskip

{\bf Acknowledgement.}
This work was supported by GENES and by the French National Research Agency (ANR) under the grants 
IPANEMA (ANR-13-BSH1-0004-02), and Labex ECODEC (ANR - 11-LABEX-0047). It was also supported by the "Chaire Economie et Gestion des Nouvelles Donn\'ees", under the auspices of Institut Louis Bachelier, Havas-Media and Paris-Dauphine.

\bibliography{bibliocontent}

\end{document}